\documentclass[12pt]{amsart}
\usepackage[margin=1.5in]{geometry}
\usepackage{amsmath, amssymb, amsfonts}
\usepackage[foot]{amsaddr}
\usepackage{color}

\usepackage{mathrsfs,extpfeil,enumerate}

\newcommand{\R}{\mathbb{R}}
\newcommand{\C}{\mathbb{C}}

\newcommand{\x}{\mbox{\boldmath $x$}}
\newcommand{\y}{\mbox{\boldmath $y$}}

\newtheorem{thm}{Theorem}
\newtheorem{prop}[thm]{Proposition}
\newtheorem{lemma}[thm]{Lemma}
\newtheorem{ex}{Example}

\begin{document}

\title{A simple construction of complex equiangular lines}

\author{Jonathan Jedwab and Amy Wiebe}
\address{Department of Mathematics, Simon Fraser University, 8888 University Drive, Burnaby, BC, Canada, V5A 1S6}
\date{28 April 2014 {(revised 3 January 2015)}}
\email{jed@sfu.ca, awiebe@sfu.ca}
\thanks{J. Jedwab is supported by an NSERC Discovery Grant. A. Wiebe was supported by an NSERC Canada Graduate Scholarship.}

\begin{abstract} A set of vectors of equal norm in $\C^d$ represents equiangular lines if the magnitudes of the inner product of every pair of distinct vectors in the set are equal. The maximum size of such a set is $d^2$, and it is conjectured that sets of this maximum size exist in $\C^d$ for every $d \geq 2$. We describe a new construction for maximum-sized sets of equiangular lines, exposing a previously unrecognized connection with Hadamard matrices. The construction produces a maximum-sized set of equiangular lines in dimensions 2, 3 and 8.
\end{abstract}

\maketitle

\section{Introduction} \label{S:Intro}

\allowdisplaybreaks

Equiangular lines have been studied for over 65 years \cite{Haan}, and their construction remains ``[o]ne of the most challenging problems in algebraic combinatorics'' \cite{Mahdad}. In particular, the study of equiangular lines in complex space has intensified recently, as its importance in quantum information theory has become apparent \cite{Appleby, Grassl05, Renes, ScottGrassl}. The main question regarding complex equiangular lines is whether the well-known upper bound (see \cite{Godsil}, for example) on the number of such lines is attainable in all dimensions: that is, whether there exist $d^2$ equiangular lines in $\C^d$ for all integers $d\ge 2$.  Zauner \cite{Zauner} conjectured 15 years ago that the answer is yes. This conjecture is supported by exact examples in dimensions 2, 3 \cite{DGS1, Renes}, 4, 5 \cite{Zauner}, 6 \cite{Grassl04}, 7 \cite{Appleby, Mahdad}, 8 \cite{Appleby28, Grassl05, Hoggar1, ScottGrassl}, 9--15 \cite{Grassl05, Grassl06, Grassl08}, 16 \cite{Appleby-squares}, 19 \cite{Appleby, Mahdad}, 24 \cite{ScottGrassl}, 28 \cite{Appleby28}, 35 and 48 \cite{ScottGrassl}, and by examples with high numerical precision in all dimensions $d\leq 67$ \cite{Renes, ScottGrassl}.

Hoggar~\cite{Hoggar1} gave a construction for $d=8$ in 1981. Although other examples in $\C^8$ have since been found \cite{Appleby28, Grassl05, ScottGrassl}, Hoggar's 64 lines are simplest to construct and can be interpreted geometrically as the diameters of a polytope in quaternionic space. 

M.~Appleby \cite{Appleby-abstract} observed in 2011: ``In spite of strenuous attempts by numerous investigators over a period of more than 10 years we still have essentially zero insight into the structural features of the equations [governing the existence of a set of $d^2$ equiangular lines in $\C^d$] which causes them to be soluble. Yet one feels that there must surely be such a structural feature \dots (one of the frustrating features of the problem as it is currently formulated is that the properties of an individual [set of $d^2$ equiangular lines in $\C^d$] seem to be highly sensitive to the dimension).'' In view of this, finding a structure for maximum-sized sets of equiangular lines that is common across multiple dimensions is highly desirable. One such example, due to Khatirinejad \cite{Mahdad}, links dimensions 3, 7 and 19 by constraining the ``fiducial vector'', from which the lines are constructed, to have all entries real.

In this paper we construct a maximum-sized set of equiangular lines in $\C^d$ from an order $d$ Hadamard matrix, for $d \in \{2, 3, 8\}$. This gives a new connection between the study of complex equiangular lines and combinatorial design theory. It also links three different dimensions $d$. Although we show that the construction in its current form cannot be extended to other values of $d$, we speculate that the construction could be modified to deal with other dimensions by using more than one complex Hadamard matrix.

The constructed set of 64 equiangular lines in $\C^8$ is of particular interest. It is \emph{almost flat}: all but one of the components of each of its 64 lines have equal magnitude. It turns out \cite{zhu} that this constructed set is equivalent to Hoggar's 64 lines under a transformation involving the unitary matrix
\begin{equation}
U =
\frac{1}{2}
\left (
\begin{array}{rrrrrrrr}
0  & 0  & i  & -1 & 0  & 0  & \phantom{-}1 & \phantom{-}i \\
0  & 0  & -1 & i  & 0  & 0  & i & 1 \\
-1 & i  & 0  & 0  & -i & -1 & 0 & 0 \\
-i & 1  & 0  & 0  & 1  & i  & 0 & 0 \\
0  & 0  & -i & 1  & 0  & 0  & 1 & i \\
0  & 0  & 1  & -i & 0  & 0  & i & 1 \\
1  & -i & 0  & 0  & -i & -1 & 0 & 0 \\
i  & -1 & 0  & 0  & 1  & i  & 0 & 0
\end{array}
\right ).
\label{EQ:unitary}
\end{equation}
(Specifically, the Hoggar lines are equivalent to a set $\{ \x_j \}$ of 64 lines in $\C^8$ given by Godsil and Roy \cite[p.~6]{GodsilRoy}, and our constructed set is $\{ \sqrt{2}U\x_j \}$.)
We believe that the description of the Hoggar lines presented here, using the basis obtained from a Hadamard matrix of order 8, is simpler than Hoggar's construction.

\section{Complex equiangular lines from Hadamard matrices} \label{S:construct}

We now introduce the main objects of study.

A line through the origin in $\C^d$ can be represented by a nonzero vector $\x\in\C^d$ which spans it. The angle between two lines in $\C^d$ represented by vectors $\x,\y$ is given by 
\begin{equation*} \arccos\left(\frac{{|\langle \x,\y\rangle|}}{||\x||\cdot||\y||}\right), \end{equation*}
where $\langle \x,\y\rangle$ is the standard Hermitian inner product in $\C^d$ and $||\x|| = \sqrt{|\langle \x, \x\rangle|}$ is the norm of $\x$. 
A set of $m\geq 2$ distinct lines in $\C^d$, represented by vectors $\x_1,\ldots,\x_m$, is {\em equiangular} if there is some real constant $c$ such that 
\begin{equation*}  \arccos\left(\frac{{|\langle \x_j,\x_k\rangle|}}{||\x_j||\cdot||\x_k||}\right) = c \hspace{10pt}\mbox{ for all } j\neq k. \label{EQ:equiang} \end{equation*}
To simplify notation, we can always take $\x_1,\ldots,\x_m$ to have equal norm, and then it suffices that there is a constant $a$ such that 
\begin{equation*} |\langle \x_j,\x_k\rangle| = a \hspace{10pt}\mbox{ for all } j\neq k. \label{EQ:equiang2} \end{equation*}

An order~$d$ {\em complex Hadamard matrix} is a $d \times d$ matrix, all of whose entries are in $\C$ and are of magnitude 1, for which
\begin{equation*} HH^{\dagger} = dI_d, \label{EQ:Hadamarddef} \end{equation*}
{where $H^\dagger$ is the conjugate transpose of $H$} 
(so that the rows of $H$ are pairwise orthogonal). If, additionally, the entries of $H$ are all in $\{1,-1\}$, then $H$ is called a {\em real Hadamard matrix} or just a Hadamard matrix. A simple necessary condition for the existence of a Hadamard matrix of order $d>2$ is that 4 divides~$d$; it has long been conjectured that this condition is also sufficient (see \cite{Horadam}, for example).

We call two complex Hadamard matrices $H,H'$ {\em equivalent} if there exist diagonal unitary matrices $D, D'$ and permutation matrices $P, P'$ 
such that
$$H' = D P H P' D'.$$

\begin{ex} Let $\omega = e^{2\pi i/3}$ and let $H, H^\prime$ be the following order $3$ complex Hadamard matrices:
\begin{align*}
H= \left( \begin {array}{ccc} 
1&1&1\\ \noalign{\medskip}
1&\omega&\omega^2\\ \noalign{\medskip}
1&\omega^2&\omega \end {array} \right)
& \hspace{25pt}
H^\prime= \left( \begin {array}{ccc} 
\omega&1&1\\ \noalign{\medskip}
1&\omega&1 \\ \noalign{\medskip}
\omega^2&\omega^2&1 \end {array} \right).
\end{align*}
Then $H$ and $H^\prime$ are equivalent, since we can obtain $H^\prime$ from $H$ by interchanging columns $2$ and $3$ and then multiplying the resulting first column and third row by $\omega$ and the resulting second row by $\omega^2$. That is,
\begin{align*}
H' & = \left( \begin {array}{ccc} 
1&0&0\\ \noalign{\medskip}
0&\omega^2&0\\ \noalign{\medskip}
0&0&\omega \end {array} \right)
 \left( \begin {array}{ccc} 
1&0&0\\ \noalign{\medskip}
0&1&0 \\ \noalign{\medskip}
0&0&1 \end {array} \right)
H
\left( \begin {array}{ccc} 
1&0&0\\ \noalign{\medskip}
0&0&1\\ \noalign{\medskip}
0&1&0 \end {array} \right)
 \left( \begin {array}{ccc} 
\omega&0&0\\ \noalign{\medskip}
0&1&0 \\ \noalign{\medskip}
0&0&1 \end {array} \right) \\
& = DPHP'D'. 
\end{align*}
\end{ex}

We now describe the main construction of this paper. Let $H$ be an order~$d$ complex Hadamard matrix. Consider $H$ to represent $d$ vectors given by the rows of the matrix. Form $d$ sets of $d$ vectors $H_1(v), \ldots, H_d(v)$ from $H$, where $H_j(v)$ is the set of vectors formed by multiplying coordinate $j$ of each vector of $H$ by the constant~$v\in\C$, and let $H(v) = \cup_{j=1}^d H_j(v)$. 

The main result of this paper, given in Theorem~\ref{THM:Hadamard}, is that in dimensions $d=2, 3, 8$ we can construct a set of $d^2$ equiangular lines in $\C^d$ as $H(v)$ for some order $d$ complex Hadamard matrix $H$ and constant $v\in\C$, and furthermore that these are the only dimensions for which this is possible. 
We firstly give examples of the construction in each of these three dimensions. Although the examples in dimensions 2 and 3 coincide with examples previously found using another construction method \cite[Section~3.4]{Zauner}, we include them here to illustrate a new connection between the three dimensions via a common construction. 

\begin{ex}
Let $H$ be the following order $2$ complex Hadamard matrix:
\begin{equation*}
H= \left( \begin {array}{rr} 
1&i\\ \noalign{\medskip}
1&-i\end {array} \right).
\end{equation*}
Then $H(v)$ consists of the following $4$ vectors 
\begin{displaymath}
\begin{array}{@{\;(}cc@{)\;}}
v & i \\ [2pt]
v & -i \\[10pt]
1 & vi \\ [2pt]
1 & -vi
\end{array}
\end{displaymath}
which are equiangular in $\C^2$ for $v = \frac{1}{2}(1+\sqrt{3})(1+i)$. \\
\label{EX:C2}
\end{ex}

\begin{ex} Let $H$ be the following order $3$ complex Hadamard matrix: 
\begin{equation*}
H = \left(\begin{array}{ccc} 1 & 1 & 1 \\ 1 & \omega & \omega^2 \\ \omega & 1 & \omega^2 \end{array}\right)
\end{equation*}
Then $H(v)$ consists of the following $9$ vectors 
\begin{displaymath}
\begin{array}{@{\;(}ccc@{)\;}}
v & 1 & 1 \\ [2pt]
v & \omega & \omega^2 \\ [2pt]
v\omega & 1 & \omega^2 \\[10pt]
1 & v & 1 \\ [2pt]
1 & v\omega & \omega^2 \\ [2pt]
\omega & v  & \omega^2 \\[10pt]
1 & 1 & v \\ [2pt]
1 & \omega & v\omega^2 \\ [2pt]
\omega & 1  & v\omega^2
\end{array}
\end{displaymath}
which are equiangular in $\C^3$ for $v=-2$. 
\label{EX:C3}
\end{ex}

\begin{ex} Let $H$ be the following order $8$ Hadamard matrix:
\begin{equation*}
H= \left( \begin {array}{rrrrrrrr} 
1&1&1&1&1&1&1&1\\ \noalign{\medskip}
1&-1&1&-1&1&-1&1&-1\\ \noalign{\medskip}
1&1&-1&-1&1&1&-1&-1\\ \noalign{\medskip}
1&-1&-1&1&1&-1&-1&1\\ \noalign{\medskip}
1&1&1&1&-1&-1&-1&-1\\ \noalign{\medskip}
1&-1&1&-1&-1&1&-1&1\\ \noalign{\medskip}
1&1&-1&-1&-1&-1&1&1\\ \noalign{\medskip}
1&-1&-1&1&-1&1&1&-1\end {array} \right).
\end{equation*}
Then $H(v)$ consists of the following $64$ vectors:
\begin{displaymath} 
\begin{array}{ @{\;(} crrrrrrr @{)\; \hspace{15pt}(\;} crrrrrrr@{)}}
v&1&1&1&1&1&1&1	&	1&v&1&1&1&1&1&1\\ 
v&-1&1&-1&1&-1&1&-1& 	1&-v&1&-1&1&-1&1&-1\\ 
v&1&-1&-1&1&1&-1&-1& 	1&v&-1&-1&1&1&-1&-1\\ 
v&-1&-1&1&1&-1&-1&1& 	1&-v&-1&1&1&-1&-1&1\\ 
v&1&1&1&-1&-1&-1&-1& 	1&v&1&1&-1&-1&-1&-1\\ 
v&-1&1&-1&-1&1&-1&1& 	1&-v&1&-1&-1&1&-1&1\\ 
v&1&-1&-1&-1&-1&1&1& 	1&v&-1&-1&-1&-1&1&1\\ 	
v&-1&-1&1&-1&1&1&-1&	1&-v&-1&1&-1&1&1&-1\\[10pt]
1&1&v&1&1&1&1&1	&	1&1&1&v&1&1&1&1\\ 
1&-1&v&-1&1&-1&1&-1& 	1&-1&1&-v&1&-1&1&-1\\ 
1&1&-v&-1&1&1&-1&-1& 	1&1&-1&-v&1&1&-1&-1\\ 
1&-1&-v&1&1&-1&-1&1& 	1&-1&-1&v&1&-1&-1&1\\ 
1&1&v&1&-1&-1&-1&-1& 	1&1&1&v&-1&-1&-1&-1\\ 
1&-1&v&-1&-1&1&-1&1& 	1&-1&1&-v&-1&1&-1&1\\ 
1&1&-v&-1&-1&-1&1&1& 	1&1&-1&-v&-1&-1&1&1\\ 	
1&-1&-v&1&-1&1&1&-1&	1&-1&-1&v&-1&1&1&-1\\[10pt]
1&1&1&1&v&1&1&1	&	1&1&1&1&1&v&1&1\\ 
1&-1&1&-1&v&-1&1&-1& 	1&-1&1&-1&1&-v&1&-1\\ 
1&1&-1&-1&v&1&-1&-1& 	1&1&-1&-1&1&v&-1&-1\\ 
1&-1&-1&1&v&-1&-1&1& 	1&-1&-1&1&1&-v&-1&1\\ 
1&1&1&1&-v&-1&-1&-1& 	1&1&1&1&-1&-v&-1&-1\\ 
1&-1&1&-1&-v&1&-1&1& 	1&-1&1&-1&-1&v&-1&1\\ 
1&1&-1&-1&-v&-1&1&1& 	1&1&-1&-1&-1&-v&1&1\\ 	
1&-1&-1&1&-v&1&1&-1&	1&-1&-1&1&-1&v&1&-1\\[10pt]
1&1&1&1&1&1&v&1	&	1&1&1&1&1&1&1&v\\ 
1&-1&1&-1&1&-1&v&-1& 	1&-1&1&-1&1&-1&1&-v\\ 
1&1&-1&-1&1&1&-v&-1& 	1&1&-1&-1&1&1&-1&-v\\ 
1&-1&-1&1&1&-1&-v&1& 	1&-1&-1&1&1&-1&-1&v\\ 
1&1&1&1&-1&-1&-v&-1& 	1&1&1&1&-1&-1&-1&-v\\ 
1&-1&1&-1&-1&1&-v&1& 	1&-1&1&-1&-1&1&-1&v\\ 
1&1&-1&-1&-1&-1&v&1& 	1&1&-1&-1&-1&-1&1&v\\ 	
1&-1&-1&1&-1&1&v&-1&	1&-1&-1&1&-1&1&1&-v\\[10pt]
\end{array} \end{displaymath}
which are equiangular in $\C^8$ for $v=-1+2i$. 
\label{EX:C8}
\end{ex}

It is easily verified by hand that each of the sets of vectors in Examples~\ref{EX:C2}, \ref{EX:C3} and \ref{EX:C8} comprises a set of $d^2$ equiangular lines in their respective dimensions. 

\section{Allowable construction parameters} \label{S:prove}

The main theorem of this paper is the following, in which we characterize all dimensions $d$, order $d$ complex Hadamard matrices $H$ and constants $v\in\C$ for which one can construct $d^2$ equiangular lines as $H(v)$. 

\begin{thm} Let $d\geq 2$. Let $H$ be an order $d$ complex Hadamard matrix and $v\in\C$ be a constant. Then $H(v)$ is a set of $d^2$ equiangular lines if and only if one of the following holds:
\begin{enumerate}
\item{$d=2$ and $v\in\left\{\frac{1}{2}(1\pm\sqrt{3})(1+i), \frac{1}{2}(1\pm\sqrt{3})(1-i),-\frac{1}{2}(1\pm\sqrt{3})(1+i)\right.$, \\ $\left.-\frac{1}{2}(1\pm\sqrt{3})(1-i)\right\},$ \label{THM:case2}}
\item{$d=3$ and $v\in\{0,-2,1\pm\sqrt{3}i\},$\label{THM:case3}}
\item{$d=8$ and $H$ is equivalent to a real Hadamard matrix and $v\in\{-1\pm2i\}$\label{THM:case8}}.
\end{enumerate}
\label{THM:Hadamard}
\end{thm}

Notice that if $H(v)$ is a set of equiangular lines then, for any complex Hadamard matrix $H'$ that is equivalent to $H$, the set $H'(v)$ is also a set of equiangular lines. This is because permutation of the rows of $H$ does not change the set $H(v)$; permutation of the columns of $H$ applies a fixed permutation to the coordinates of all vectors of $H(v)$; multiplication of a row of $H$ by a constant $c \in \C$ of magnitude 1 multiplies one vector in each set $H_j(v)$ by~$c$; and multiplication of a column of $H$ by a constant $c \in \C$ of magnitude 1 multiplies a fixed coordinate of each vector of $H(v)$ by~$c$. In each case, the magnitude of the inner product between pairs of distinct vectors in $H(v)$ is unchanged.

There are only three types of inner product that can arise between distinct vectors of $H(v)$:
\begin{enumerate}[(i)]
\item{the inner product of two distinct vectors within a set $H_j(v)$, \label{EQ:inner1block}}
\item{the inner product of two vectors of distinct sets $H_j(v), H_k(v)$ which are derived from the same row of $H$, \label{EQ:inner1row}}
\item{the inner product of two vectors of distinct sets $H_j(v), H_k(v)$ which are derived from distinct rows of $H$. \label{EQ:inner2blocks2rows}}
\end{enumerate}
Therefore $H(v)$ is a set of $d^2$ equiangular lines if and only if the equations obtained by equating the magnitudes of every inner product of type (\ref{EQ:inner1block}), (\ref{EQ:inner1row}) and (\ref{EQ:inner2blocks2rows}) have a solution. In Lemma~\ref{LEM:Hadamardconstant} we show that only one magnitude occurs for all inner products of type (i) and likewise for all inner products of type~(ii). In Propositions~\ref{PROP:C2}--\ref{PROP:C3} and Theorem~\ref{THM:real} we show that, for dimensions 2, 3 and 8, inner products of type (iii) take values in only a small set. It is then straightforward to characterize the solutions obtained by equating magnitudes, for these three dimensions, and to show that the corresponding equations have no solutions in other dimensions. This establishes Theorem~\ref{THM:Hadamard}.

\begin{lemma} 
Let $v=a+ib$ for $a,b\in\R$. For all $d$, every inner product of type (\ref{EQ:inner1block}) has magnitude $|a^2+b^2-1|$ and every inner product of type (\ref{EQ:inner1row}) has magnitude  $|2a+d-2|$.
\label{LEM:Hadamardconstant}
\end{lemma}

\begin{proof} Two vectors having inner product of type (\ref{EQ:inner1block}) are derived from distinct rows of $H$; thus their inner product in $H$ is 0. 
In $H(a+ib)$ these vectors are multiplied in the same coordinate by $a+ib$, giving an inner product of magnitude $|a^2+b^2-1|$.

Two vectors having inner product of type (\ref{EQ:inner1row}) are derived from the same row of $H$; thus their inner product in $H$ is $d$ (given by a contribution of 1 from each coordinate of the vectors).
In $H(a+ib)$ these vectors are multiplied in different coordinates by $a+ib$, giving an inner product of magnitude $|2a+d-2|$.
\end{proof}

\begin{prop}[$d=2$]
Let $H$ be an order $2$ complex Hadamard matrix. Then $H(v)$ is a set of $4$ equiangular lines in $\C^2$ if and only if $v\in\left\{\frac{1}{2}(1\pm\sqrt{3})(1+i)\right.$, $\left.\frac{1}{2}(1\pm\sqrt{3})(1-i), -\frac{1}{2}(1\pm\sqrt{3})(1+i),-\frac{1}{2}(1\pm\sqrt{3})(1-i)\right\}$.  
\label{PROP:C2}
\end{prop}

\begin{proof} 
Up to equivalence, the only order 2 complex Hadamard matrix \cite[Prop.~2.1]{Haagerup} is
$$H = \left(\begin{array}{cc} 1  & 1 \\ 1 & -1 \end{array}\right).$$ 
Both inner products of type (\ref{EQ:inner2blocks2rows}) that occur in $H(a+ib)$ (where $a, b \in \R$) have magnitude $|(a+ib)-\overline{(a+ib)}| = |2b|$. 

Using Lemma~\ref{LEM:Hadamardconstant}, $H(a+ib)$ is therefore a set of equiangular lines if and only if we can solve the equations
$$|a^2+b^2-1| = |2a| = |2b| \quad \mbox{for $a, b \in \R$}.$$
This can be done exactly when $a\in\left\{\frac{1}{2}(1\pm\sqrt{3}), -\frac{1}{2}(1\pm\sqrt{3})\right\}$ and $b = \pm a$. 
\end{proof}

\begin{prop}[$d=3$]
Let $H$ be an order $3$ complex Hadamard matrix.  Then $H(v)$ is a set of $9$ equiangular lines in $\C^3$ if and only if $v\in\{0,-2,1\pm\sqrt{3}i\}$.
\label{PROP:C3}
\end{prop}

\begin{proof}
Let $\omega=e^{2\pi i/3}$. Up to equivalence, the only order 3 complex Hadamard matrix \cite[Prop.~2.1]{Haagerup} is
$$H = \left(\begin{array}{ccc} 1  & 1 & 1 \\ 1 & \omega & \omega^2 \\ 1 & \omega^2 & \omega \end{array}\right).$$
All inner products of type (\ref{EQ:inner2blocks2rows}) that occur in $H(v)$ are derived from rows of $H$ having inner product $1+\omega+\omega^2=0$. In $H(v)$, each of these inner products takes the form $\omega^j(v+\overline{v}\omega+\omega^2)$ or $\omega^j(v+\omega+\overline{v}\omega^2)$ for some $j\in\{0,1,2\}$. For $v = a+ib$ with $a, b \in \R$, these inner products have magnitude $|a-1+b\sqrt{3}|$ and $|a-1-b\sqrt{3}|$, respectively, and both magnitudes occur.

Using Lemma~\ref{LEM:Hadamardconstant}, $H(a+ib)$ is therefore a set of equiangular lines if and only if we can solve the equations
$$|a^2+b^2-1| = |2a+1| = |a-1+b\sqrt{3}|=|a-1-b\sqrt{3}| \quad \mbox{for $a, b \in \R$}.$$
This can be done exactly when $(a,b)\in\{(0,0),(-2,0)$, $(1,\pm\sqrt{3})\}$. 
\end{proof}

We now complete the proof of our main result, by showing that if $H(v)$ is a set of $d^2$ equiangular lines for $d>3$ then we must be in case (\ref{THM:case8}) of Theorem~\ref{THM:Hadamard}.

\begin{thm} Let $d>3$ and let $H$ be an order $d$ complex Hadamard matrix. Then $H(v)$ is a set of $d^2$ equiangular lines if and only if $d=8$ and $H$ is equivalent to a real Hadamard matrix and $v\in\{-1\pm2i\}$.
\label{THM:real}
\end{thm}

\begin{proof}Let $H = (h_{jk})$ be an order $d$ complex Hadamard matrix. We consider two cases. 

Case 1 is where, for every pair of distinct rows of $H$, all summands of the inner product of the rows take values in a set $\{\xi, -\xi\}$ for some $\xi \in \C$ (depending on the row pair) of magnitude~1.
We now show that, in this case, $H$ is equivalent to some real Hadamard matrix $H'$. 
Firstly transform each entry of the first row of $H$ to be 1, by multiplying each column~$k$ of $H$ by the constant $h_{1k}^{-1}\in\C$ of magnitude 1.
Then all summands of the inner product of the resulting rows~1 and~$j$ take values in a set $\{\xi_j, -\xi_j\}$ for some $\xi_j \in \C$ of magnitude~1, and so all entries of row $j$ lie in $\{\xi_j, -\xi_j\}$. Multiply each row $j$ by the constant $\xi_j^{-1} \in \C$ of magnitude~1 to obtain a real Hadamard matrix~$H'$.

We next show that all inner products of type~(\ref{EQ:inner2blocks2rows}) that occur in $H'(v)$ have one of exactly two magnitudes.
All such inner products are derived from rows of $H'$ having inner product $\frac{d}{2}(1) + \frac{d}{2}(-1)$, and $\frac{d}{2} \ge 2$ since $d > 3$. In $H'(v)$, each of these inner products takes the form $(-1)^j \left(v+\overline{v} + (\frac{d}{2}-2)(1)+\frac{d}{2}(-1)\right) = (-1)^j (v + \overline{v} -2)$ or $(-1)^j \left(v-\overline{v} + (\frac{d}{2}-1)(1)+(\frac{d}{2}-1)(-1)\right) = (-1)^j (v-\overline{v})$ for some $j \in \{0,1\}$. For $v = a+ib$ with $a,b\in\R$, these inner products have magnitude $|2a-2|$ and $|2b|$, respectively, and both magnitudes occur.

Using Lemma~\ref{LEM:Hadamardconstant}, $H(a+ib)$ is therefore a set of equiangular lines if and only if we can solve the equations
$$|a^2+b^2-1| = |2a+d-2| = |2a-2| = |2b| \quad \mbox{for $a,b\in\R$}.$$
This can be done exactly when $d=8$ and $(a,b)\in\{(-1,\pm2)\}$. 

Case 2 is where the summands of the inner product of some pair of distinct rows of $H$ are $\xi_1$, $\xi_2$, $\xi_3$ for distinct $\xi_j \in \C$ of magnitude 1, together with $\eta_1, \dots, \eta_{d-3}$ for $d-3 > 0$ other elements $\eta_j \in \C$ of magnitude~1 (not necessarily distinct from each other or from the $\xi_j$). 
Thus in $H(v)$ there are three pairs of vectors, derived from this pair of rows, having inner products $(v-1)\eta_1+(\overline{v}-1)\xi_1$, $(v-1)\eta_1+(\overline{v}-1)\xi_2$ and $(v-1)\eta_1+(\overline{v}-1)\xi_3$. 

We now show that there is no $a,b \in \R$ for which $H(a+ib)$ is a set of equiangular lines. Suppose otherwise, for a contradiction. Then the above three inner products have equal magnitude for $v = a+ib$, so that
\begin{align}
|(a-1+ib)\eta_1+(a-1-ib)\xi_1|  & = |(a-1+ib)\eta_1+(a-1-ib)\xi_2| \nonumber \\
				& = |(a-1+ib)\eta_1+(a-1-ib)\xi_3|. \label{EQ:innerprodMatt}
\end{align}
Notice that from Lemma~\ref{LEM:Hadamardconstant}, we cannot have $(a,b)=(1,0)$ as this would imply $d=0$. Therefore $(a-1+ib)\eta_1\neq0$ and $\sqrt{(a-1)^2+b^2}\neq 0$. Now the $\xi_j$ are all distinct with magnitude $1$, so $(a-1-ib)\xi_1$, $(a-1-ib)\xi_2$, $(a-1-ib)\xi_3$ are all distinct with magnitude $\sqrt{(a-1)^2+b^2}>0$. But then it is clear geometrically that only two of these three quantities can have equal magnitude after being added to $(a-1+ib)\eta_1 \ne 0$. This contradicts~\eqref{EQ:innerprodMatt}.

\end{proof}

\section*{Acknowledgements}

We thank Matt DeVos for his interest in this construction and the resulting helpful discussion and important insight regarding the proof of Theorem~\ref{THM:real}. We are grateful to Huangjun Zhu for his generosity in pointing out the unitary transformation involving~\eqref{EQ:unitary}.


\begin{thebibliography}{10}

\bibitem{Appleby}
D.~M. Appleby, \emph{Symmetric informationally complete-positive operator
  valued measures and the extended {C}lifford group}, J. Math. Phys.
  \textbf{46} (2005), 052107.

\bibitem{Appleby28}
D.~M. Appleby, I.~Bengtsson, S.~Brierley, {\AA}.~Ericsson, M.~Grassl, and
  J.-{\AA}. Larsson, \emph{Systems of imprimitivity for the {C}lifford group},
  Quantum Inf. Comput. \textbf{14} (2014), no.\ 3\&4, 339--360.

\bibitem{Appleby-squares}
D.~M. Appleby, I.~Bengtsson, S.~Brierley, M.~Grassl, D.~Gross, and J.-{\AA}.
  Larsson, \emph{The monomial representations of the {C}lifford group}, Quantum
  Inf. Comput. \textbf{12} (2012), no.\ 5\&6, 404--431.

\bibitem{Appleby-abstract}
M. Appleby, abstract for \emph{SIC-POVMs, Theta Functions and Squeezed States}, 2010-2011 Clifford Lectures, Tulane University, http://tulane.edu/sse/math/news/clifford-lectures-2011.cfm.

\bibitem{DGS1}
P.~Delsarte, J.~M. Goethals, and J.~J. Seidel, \emph{Bounds for systems of
  lines, and {J}acobi polynomials}, Philips Res. Repts \textbf{30} (1975),
  91--105.

\bibitem{Godsil}
C.~Godsil, \emph{Quantum geometry: {MUB}'s and {SIC-POVM}'s}, 2009,
  http://quoll.uwaterloo.ca/pdfs/perth.pdf.
  
\bibitem{GodsilRoy}
C.~Godsil and A.~Roy \emph{Equiangular lines, mutually unbiased bases, and spin models}, 
  European J. Combin. \textbf{30} (2009), no.~1, 246--262.

\bibitem{Grassl04}
M.~Grassl, \emph{On {SIC}-{POVM}s and {MUB}s in dimension 6}, Proceedings ERATO
  Conference on Quantum Information Science (Tokyo), 2004, pp.~60--61.

\bibitem{Grassl05}
\bysame, \emph{Tomography of quantum states in small dimensions}, Proceedings
  of the {W}orkshop on {D}iscrete {T}omography and its {A}pplications
  (Amsterdam), Electron. Notes Discrete Math., vol.~20, Elsevier, 2005,
  pp.~151--164.

\bibitem{Grassl06}
\bysame, \emph{Finding equiangular lines in complex space}, 2006, MAGMA 2006
  Conference, Technische Universit\"at Berlin.

\bibitem{Grassl08}
\bysame, \emph{Computing equiangular lines in complex space}, Mathematical
  {M}ethods in {C}omputer {S}cience, Lecture Notes in Comput. Sci., vol. 5393,
  Springer, Berlin, 2008, pp.~89--104.

\bibitem{Haagerup}
U.~Haagerup, \emph{Orthogonal maximal abelian {$*$}-subalgebras of the
  {$n\times n$} matrices and cyclic {$n$}-roots}, Operator algebras and quantum
  field theory ({R}ome, 1996), Int. Press, Cambridge, MA, 1997, pp.~296--322.

\bibitem{Haan}
J.~Haantjes, \emph{Equilateral point-sets in elliptic two- and
  three-dimensional spaces}, Nieuw Arch. Wiskunde (2) \textbf{22} (1948),
  355--362.

\bibitem{Hoggar1}
S.~G. Hoggar, \emph{Two quaternionic {$4$}-polytopes}, The Geometric Vein,
  Springer, New York, 1981, pp.~219--230.

\bibitem{Horadam}
K.~J. Horadam, \emph{Hadamard Matrices and their Applications}, Princeton
  University Press, Princeton, NJ, 2007.

\bibitem{Mahdad}
M.~Khatirinejad, \emph{On {W}eyl-{H}eisenberg orbits of equiangular lines}, J.
  Algebraic Combin. \textbf{28} (2008), 333--349.

\bibitem{Renes}
J.M. Renes, R.~Blume-Kohout, A.~J. Scott, and C.M. Caves, \emph{Symmetric
  informationally complete quantum measurements}, J. Math. Phys. \textbf{45}
  (2004), no.~6, 2171--2180.

\bibitem{ScottGrassl}
A.~J. Scott and M.~Grassl, \emph{Symmetric informationally complete
  positive-operator-valued measures: a new computer study}, J. Math. Phys.
  \textbf{51} (2010), 042203.

\bibitem{Zauner}
G.~Zauner, \emph{Quantendesigns - grundz\"uge einer nichtkommutativen
  designtheorie}, Ph.D. thesis, University of Vienna, 1999.

\bibitem{zhu}
H.~Zhu, personal communication, June 2014.

\end{thebibliography}

\def\polhk#1{\setbox0=\hbox{#1}{\ooalign{\hidewidth
  \lower1.5ex\hbox{`}\hidewidth\crcr\unhbox0}}} \def\cprime{$'$}
\providecommand{\bysame}{\leavevmode\hbox to3em{\hrulefill}\thinspace}
\providecommand{\MR}{\relax\ifhmode\unskip\space\fi MR }
\providecommand{\MRhref}[2]{%
  \href{http://www.ams.org/mathscinet-getitem?mr=#1}{#2}
}
\providecommand{\href}[2]{#2}

\end{document}